\documentclass[11pt]{article}
\usepackage[top=3cm,bottom=3.5cm,inner=3cm,outer=3cm]{geometry}

\usepackage{geometry}
\usepackage{amsthm} 
\usepackage{amsmath}
\usepackage{amssymb}
\usepackage{bbm}
\usepackage{mdwlist}
\usepackage{thmtools}
\usepackage{thm-restate}
\usepackage{setspace}

\newtheorem{thm}{Theorem} 
\newtheorem{lem}[thm]{Lemma}

\newtheorem{claim}[thm]{Claim}

\newtheorem{ex}[thm]{Example}

\newtheorem{question}[thm]{Question}
\newtheorem*{question*}{Question}

\renewenvironment{proof}[1][]{\begin{trivlist}
\item[\hspace{\labelsep}{\bf\noindent Proof#1.\/}] }{\qed\end{trivlist}}

\DeclareMathOperator{\tr}{tr}

\newcommand{\A}{\mathcal{A}}
\newcommand{\B}{\mathcal{B}}
\newcommand{\C}{\mathcal{C}}
\newcommand{\D}{\mathcal{D}}
\newcommand{\E}{\mathcal{E}}

\newcommand{\eps}{\epsilon}
\newcommand{\lam}{\lambda}

\newcommand{\eqS}{\: = \:}
\newcommand{\leS}{\: \le \:}

\newcommand{\geS}{\: \ge \:}

\newcommand{\addspace}[1]{
\addtolength{\jot}{4pt}
#1
}

\title{Eigenvalues of subgraphs of the cube}
\author{B\'ela Bollob\'as \thanks{
        Department of Pure Mathematics and Mathematical Statistics, 	
        University of Cambridge, 
        Wilberforce Road, 
        CB3\;0WB Cambridge, 
        UK;
        e-mail:
        \mbox{\{\texttt{B.Bollobas,S.Letzter}\}\texttt{@dpmms.cam.ac.uk}}\,.
    }
    \thanks{
        Department of Mathematical Sciences, 
        University of Memphis, 
        Memphis TN 38152, 
        USA
        and
        London Institute for Mathematical Sciences, 
        35a South St., 
        Mayfair, 
        London W1K 2XF,
        UK.
    }
    \thanks{
        Research supported in part by NSF grant DMS-1301614 and EU MULTIPLEX
        grant 317532.
    }
    \and
    Jonathan Lee \thanks{
        Mathematical Institute, 
        University of Oxford, 
        Woodstock Road, 
        Oxford, 
        OX2 6GG, 
        United Kingdom.
        e-mail: 
        \mbox{\texttt{jonathan.lee@maths.ox.ac.uk}}
    }
    \and
    Shoham Letzter \footnotemark[1]
}

\begin{document}
\maketitle
\begin{abstract}
    \setlength{\parindent}{0in} 
    \setlength{\parskip}{.08in} 
    \noindent
    We consider the problem of maximising the largest eigenvalue of subgraphs of
    the hypercube $Q_d$ of a given order.  We believe that in most cases,
    Hamming balls are maximisers, and our results support this belief.  We show
    that the Hamming balls of radius $o(d)$ have largest eigenvalue that is
    within $1 +o(1)$ of the maximum value. We also prove that Hamming balls
    with fixed radius maximise the largest eigenvalue exactly, rather than
    asymptotically, when $d$ is sufficiently large.  Our proofs rely on the
    method of compressions. 
    \setlength{\parskip}{.1in} 
\end{abstract}

\section{Introduction}
    In the last few decades much research has been done on spectra of graphs,
    i.e.~the eigenvalues of the adjacency matrices of graphs; see Finck and
    Grohmann~\cite{FG},  Hoffman~\cite{hof1, hof2}, Nosal~\cite{nos},
    Cvetkovi\'c,  Doob and Sachs~\cite{CDS}, Neumaier~\cite{Neum}, Brigham and
    Dutton~\cite{brig1, brig2}, Brualdi and Hoffman~\cite{BH},
    Stanley~\cite{stan}, Shearer~\cite{she}, Powers~\cite{pow}, Favaron,
    Mah\'eo and Sacl\'e~\cite{FMS1, FMS2}, Hong~\cite{hong}, Liu, Shen and
    Wang~\cite{LSW},  Nikiforov~\cite{nik-walks, nik-bounds1, nik-bounds2,
    nik-extr}, and Cvetkovi\'c, Rowlinson and  Simi\'c~\cite{CRS} for a small
    selection of relevant publications. Perhaps the most basic property of the
    spectrum of a graph is its radius, i.e.~the maximal eigenvalue: this has
    received especially much attention. Here we shall mention a small handful of
    these results.

    In what follows, $A(G)$ denotes the adjacency matrix of a graph $G$ and
    $\lam_1(G)$ denotes the largest eigenvalue of $A(G)$. As usual, we write
    $e(G)$ for the number of edges, $\Delta(G)$ for the maximal degree and
    $\overline{d}(G)$ for the average degree. Trivially, $\overline{d}(G) \le
    \lambda_1(G) \le \Delta(G)$; in particular, if $G$ is $d$-regular then
    $\lambda_1(G)=d$. In 1985, Brualdi and Hoffman~\cite{BH} gave an upper bound
    on $\lambda_1(G)$ in terms of $e(G)$: if $e(G)\le \binom{k}{2}$ for some
    integer $k\ge 1$ then $\lambda_1(G) \le k-1$, with equality iff $G$ consists
    of a $k$-clique and isolated vertices. Extending this result,
    Stanley~\cite{stan} showed that if $e(G)=m$ then $\lambda_1(G)\le
    \frac{1}{2} \left(-1+\sqrt{8m+1}\right)$, with equality only as before. In
    1993, Favaron, Mah\'eo and Sacl\'e~\cite{FMS2} published an upper bound on
    $\lambda_1(G)$ in terms of the local structure of $G$: writing $s(G)$ for
    the maximum of the sum of degrees of vertices adjacent to some vertex, we
    have $\lambda_1(G)\le \sqrt{s(G)}$. Furthermore, if $G$ is connected then
    equality holds iff $G$ is regular or bipartite semi-regular (i.e.~vertices
    in the same class have equal degrees). In particular, if $G$ is a
    triangle-free graph with $m$ edges then $s(G)\le m$, so $\lambda_1(G)\le
    \sqrt{m}$. This inequality was first proved by Nosal~\cite{nos} in 1970. The
    star $K_{1,m}$ shows that this inequality is best possible. 

    Our main aim in this paper is to study the maximal eigenvalue of induced
    subgraphs of the cube $Q_d$ on $2^d$ vertices, rather than general graphs
    restricted by their parameters like order and size. To be precise, our aim
    is to give a partial answer to the following question posed by
    Fink~\cite{fink} and in a weaker form by Friedman and Tillich \cite{ft}. 
    \begin{question}
        Given $m$, $1\le m \le 2^d$, what is the maximum of
        the maximal eigenvalue of $Q_d[U]$, where $|U|=m$?
    \end{question}
    This problem can be
    viewed as a variant of the `classical' isoperimetric problem in the cube.
    Indeed, since $Q_d$ is $d$-regular, the problem of bounding the maximal
    eigenvalue of the subgraph $Q_d[U]$ of $Q_d$ induced by a set $U \subset
    V(Q_d)=\{0,1\}^d$ is closely related the the size of the edge boundary of
    $U$, the set of edges joining a vertex in $U$ to one not in $U$. If the
    maximal eigenvalue of $Q_d[U]$ is $\lambda_1$, then $e\left(Q_d[U]\right)
    \le \lambda_1 |U|/2$, so the size of the edge boundary of $U$ is at least
    $(d-\lambda_1)|U|$. Thus, if $\lambda_1\le \lambda(m)$ whenever $|U|=m$,
    then for every set of $m$ vertices of the cube $Q_d$ the edge boundary has
    size at least $(d-\lambda(m))m$.

    The study of eigenvalues as a form of isoperimetric inequality is not new:
    in 1985, Alon and Milman \cite{alon-milman} showed that there is a close
    relation between the second smallest eigenvalue of the Laplacian of a graph
    and some expansion properties of the graph. The nature of our problem is
    very different from this. A vaguely related problem has been studied by
    Reeves, Farr, Blundell, Gallagher and Fink~\cite{reeves}.

    Before we state our results, we give some precise definitions. Our ground
    graph is taken to be $Q_d$, the \emph{$d$-dimensional hypercube}, where the
    vertices are labelled by the $0, 1$ strings of length $d$, so that
    $V(Q_d)=\{0,1\}^d$. Two vertices are connected by an edge if they differ in
    exactly one coordinate.  We shall often use the obvious correspondence
    between binary strings of length $d$ and subsets of $[d]$ in which a subset
    corresponds to its characteristic function.  A \emph{subcube} of $Q_d$ of
    dimension $i$ is the graph induced by a subset of the vertices obtained by
    fixing the values of all but $i$ coordinates.  The \emph{Hamming ball}
    $H_d^i$ is the subgraph of $Q_d$ induced by the vertices with at most $i$
    ones in their strings.  We note that the subgraphs minimising the sizes of
    the vertex and edge boundaries among all subgraphs of $Q_d$ with a given
    order are well known.  In particular, Harper (see \cite{isoperimetric1} and
    \cite{isoperimetric2}) showed in 1966 that the Hamming balls  minimise the
    size of the vertex boundary among subgraphs of the same order. In 1976, Hart
    \cite{isoperimetric3} proved a similar result, showing that subcubes
    minimise the size of the edge boundary among subgraphs of the hypercube of
    the same order. 

    As the problem of maximising $\lam_1$ is a form of an isoperimetric problem,
    it seems natural to believe that either Hamming balls or subcubes should be
    maximisers of $\lam_1$. Despite the connection between $\lam_1$ and the edge
    boundary, we believe that in many cases, the task of maximising $\lam_1$ is
    related to minimising the vertex boundary.  More precisely, we believe that
    for most radii sufficiently smaller than $d/2$, Hamming balls maximise $\lam_1$. 
    
    We prove several results in this direction.  Our first
    result, which is relatively easy, gives a precise answer when the number of
    vertices is at most the dimension of the hypercube.
    \begin{restatable}{thm}{thmStar} \label{thm:star}
        Let $G$ be an induced subgraph of $Q_d$ with $n\le d$ vertices. Then for
        $n\ge 103$, $\lambda_1(G)\le \sqrt{n-1}$ with equality if and only if
        $G$ is a star.
    \end{restatable}
    We note that the conclusion of Theorem \ref{thm:star} does not hold for all
    $n$.  Indeed, for $n = 4$, the largest eigenvalue of $Q_2$ (or $C_4$) is
    $2$, which is larger than $\sqrt{3}$, the largest eigenvalue of the star
    $K_{1, 3}$.

    In order to obtain more general results we evaluate the largest eigenvalue
    of the Hamming ball $H_d^i$ for radii tending to infinity with the dimension
    of the cube. 
    \begin{restatable}{thm}{thmHammingEval}\label{thm:hamming-eval} 
        If $d, i \rightarrow \infty$ and $i \leq \frac{d}{2}$ then
        $$\lambda_1(H_d^i) \eqS 2\sqrt{i(d+1-i)} \left(
        1+O\left(i^{-\frac{1}{2}}\log^\frac{1}{2}i\right)\right).$$
    \end{restatable}

    Our first main result is a generalisation of Theorem \ref{thm:star}. We
    prove that for a wide range of radii, the Hamming balls have largest eigenvalues
    which are asymptotically largest among all subgraph of the cube of the same
    order.  We note that Samorodnitsky \cite{samo} obtained an equivalent result
    for a wider range of radii (namely for radii $i$ satisfying $i \to \infty$;
    our proof works also if $i$ is bounded). His proof
    methods are very different from ours.
    \begin{restatable}{thm}{thmHammingBestAsymptotically}
            \label{thm:hamming-best-asym}
        Let $i=i(d)=o(d)$ and let $G$ be a subgraph of $Q_d$ with
        $n=O\left(\left|[d]^{(\le i)}\right|\right)$ vertices.  Then
        $\lambda_1(G)\le (1+o(1)) \;
        \lambda_1(H^i_d)$.
    \end{restatable}

    Finally, our second main result gives an exact answer when the radius is
    fixed. 
    \begin{restatable}{thm}{thmFixedRadius} \label{thm:fixed-radius}
        For every $i$ there is $d_0=d_0(i)$ such that for $d\ge d_0$ the Hamming
        ball $H_d^i$ maximises the largest eigenvalue among subgraphs of $Q_d$
        with the same number of vertices.
    \end{restatable} 

    In the next section, Section \ref{sec:compressions}, we state and prove 
    results about compressions which will be used in the proofs of the above
    theorems.  We prove Theorem \ref{thm:star} in Section \ref{sec:star}.  In
    Section \ref{sec:calc-eval} we prove Theorem \ref{thm:hamming-eval} as well
    as other bounds on the largest eigenvalue of certain subgraphs of the
    cube.  We prove our first main result, Theorem
    \ref{thm:hamming-best-asym}, in Section \ref{sec:asym} and our
    second main result, Theorem \ref{thm:fixed-radius}, is proved in Section
    \ref{sec:fixed-radius}. We conclude with some remarks and open questions in
    Section \ref{sec:conclusion}.

\section{Compressions} \label{sec:compressions}
    In this sections we prove the results that we shall need about compressions.
    We start by introducing notation.  Let $v \in \mathbb{R}^{V(G)} \subseteq
    \mathbb{R}^{V(Q_d)}$. Then $\langle A(G)v, v\rangle = \langle A(Q_d)v,
    v\rangle$, since the support of $v$ is contained in $V(G)$. Hence
    \[\max_{|G| = n}\lambda_1(G) = \max_{|G| = n,\,\, ||v|| = 1}\langle A(G)v,
    v\rangle = \max_{||v|| = 1,\,\,\textrm{supp}(v) = n}\langle A(Q_d)v,
    v\rangle.\]

    We consider a notion of \emph{compressions} acting on vectors in
    $\mathbb{R}^{V(Q_d)}$. Let $U, V \subseteq [d]$ be such that $U \cap V = \emptyset$
    and let $v \in \mathbb{R}^{A(Q_d)}$.  We define $C_{U, V}(v) \in
    \mathbb{R}^V(Q_d)$ as follows, where $S \subseteq [d]$.
    \begin{align*}
    \left(C_{U, V}(v)\right)_S = 
    \left\{
    \begin{array}{ll}
    \max(v_S, v_{S \triangle (U \cup V)}) & V \subseteq S \text{ and } U \cap S = \emptyset \\
    \min(v_S, v_{S \triangle (U \cup V)}) & U \subseteq S \text{ and } V \cap S = \emptyset \\
    v_S & \text{otherwise}
    \end{array}
    \right.
    \end{align*}

    Note that $C_{U,V}$ applies a $U-V$ compression to the support of $v$,
    leaving the multiset of entries of $v$ unchanged. In particular, it
    preserves the size of the support of $v$ and its norm.

    The \emph{binary order} on $Q_d$ is defined as follows: $S < T$ if
    and only if $\max{S \triangle T} \in T$ for $S, T \in V(Q_d)$.  We define
    the binary $i$-compression $C_i(v)$ to rearrange the
    values $(v_S)_{i \in S}$ to be decreasing in the binary order
    restricted to the subcube $\{S : i \in S\}$, and rearrange the values
    $(v_S)_{i \notin S}$ to be decreasing in the binary order
    restricted to $\{S : i \notin S\}$. We define $C^+_i $ and $C^-_i$ to be the
    restrictions of $v$ to sets containing $i$ or not containing $i$
    respectively. Note that $C^+_i$ and $C^-_i$ commute with
    $C_i$.

    We may naturally apply these maps to the indicator function of a set $F$ to
    obtain another indicator function, coinciding with the usual definitions of
    these maps on sets. We suppress explicit usage of the indicator function
    where this can be done without confusion.

    Given $i \in [d]$, we abuse notation by denoting the singleton
    $\{i\}$ by $i$ where this is not likely to cause confusion. Furthermore, if
    $S \subseteq [d]$ we denote $S \cup \{i\}$ by $S + i$ and similarly we
    denote $S \setminus \{i\}$ by $S - i$.
    The following two results show that by applying a $C_{i, \emptyset}$
    compression or a $C_{i, j}$ compression to a vector $v$, we do not decrease
    the inner product $\langle A(Q_d) v, v \rangle$.

    \begin{lem}\label{thm:down-closed}
        Let $i \in [d]$ and $v \in \mathbb{R}^{V(Q_d)}$ and denote $A = A(Q_d)$
        and $\bar{v} = C_{i, \emptyset}(v)$.  Then $\langle Av, v \rangle \le
        \langle A\bar{v}, \bar{v} \rangle$.
    \end{lem}
    \begin{proof}
        Consider an edge $ST \in E(Q_d)$ with $S \subset T$. If $T\setminus S =
        \{i\}$, then $v_S$ and $v_T$ are either swapped or not, and in either case
        the contribution of $ST$ to the inner product is unchanged. All other edges
        have either $i \in S \cap T$ or $i \neq S \cup T$. These edges 
        come in pairs $(S, S+j)$, $(S+i, S+i+j)$. By the rearrangement inequality
        and the definition of $C_{i,\emptyset}$, the contribution of this pair of
        edges to the inner product is larger in $C_{i,\emptyset}(v)$ than in $v$. 
    \end{proof}

    \begin{lem}\label{lem:ij-compressed}
        Let $i, j \in [d]$ be distinct, let $v \in \mathbb{R}^{V(Q_d)}$, and
        denote $A = A(Q_d)$ and $\bar{v} = C_{i, j}(v)$.
        Then $\langle Av, v \rangle \leq \langle A \bar{v}, \bar{v}\rangle$.
    \end{lem}
    \begin{proof}
        Consider an edge $ST \in E(Q_d)$, $S \subseteq T$. The function $C_{i,
        j}$ is a composition of conditional swaps, and each vertex of $Q_d$ is
        involved in at most one of these swaps. If neither $S$ nor $T$ are
        involved in a swap, then the contribution of the edge $ST$ to the inner
        product is unchanged.

        If both $S$ and $T$ are involved in a swap, then if $i \in S$ we have
        $v_S$ potentially being swapped with $v_{S-i+j}$ and $v_T$ potentially
        being swapped with $v_{T-i+j}$; if $i\notin S$ then $j \in S$, so $v_S$
        and $v_T$ are potentially swapped with $v_{S-j+i}$ and $v_{T-j+i}$
        respectively. Hence edges $ST$ where both vertices are potentially
        swapped come in pairs $(S,T)$, $(S-i+j, T-i+j)$. By the rearrangement
        inequality, the contribution of each of these pairs to the inner product
        is increased by $C_{i, j}$.

        If only $S$ is involved in a swap, then exactly one of $i$ and $j$,
        whilst both are in $T$. Hence such edges come in pairs $(T-i, T)$ and
        $(T-j, T)$, and the contribution of such pairs to the inner product is
        unchanged by $C_{i, j}$. Similarly, the edges where only $T$ is involved
        in a swap come in pairs $(S, S+i)$ and $(S, S+j)$, and the contribution
        of such pairs to the inner product is unchanged by $C_{i, j}$.
    \end{proof}

    We say that a vector $v \in \mathbb{R}^{V(Q_d)}$ is \emph{compressed} if
    $C_{U, \emptyset}(v) = v$ for every $U \subseteq [d]$ and $C_{i, j}(v) = v$
    for $1\le i < j \le d$.  It follows from Lemma \ref{thm:down-closed} and
    \ref{lem:ij-compressed} that in order to find the maximum of $\lam_1(G)$
    over subgraphs of the cube of order $n$, it suffices to consider induced
    graphs $G$ whose vertex set is compressed.  Furthermore, this maximum equals
    the maximum of $\langle Av, v \rangle$ over compressed vectors $v$ with
    support of size $n$.

    \subsection{Counting copies of subcubes}
        The aim of this subsection is to provide an upper bound on the number of
        copies of a subcube in a subgraph $G$ of the cube in terms of $|G|$.

        Given a set $U \subseteq V(Q_d)$ and $d' \le d$ we denote the number of
        copies of $Q_{d'}$ in $Q_d[U]$ by $\#(Q_{d'} \subseteq U)$.  The
        following result, which was proved by Bollob\'as and Radcliffe
        \cite{bollobas-radcliffe}, shows the number of copies of $Q_{d'}$ is
        maximised by initial segments of the binary order. We present a proof
        here
        for the sake of completeness.  
        \begin{lem}\label{lem:subcubes} 
            Let $U, I \subseteq V(Q_d)$ with $|U| = |I|$ and $I$ is an initial
            segment in binary order.  Then for any $d' \leq d$,
            $$\#(Q_{d'} \subseteq U) \leS \#(Q_{d'} \subseteq I).$$
        \end{lem}
        \begin{proof}
            We prove the lemma by induction on $d'$. $d' = 0$ is trivial, as $|U| =
            |I|$ Suppose that $d' > 0$. We proceed by induction on $d \geq d'$.
            For $d = d'$ we have that both $\#(Q_{d'} \subseteq U)$ and
            $\#(Q_{d'} \subseteq I)$ are $0$ if $|U| = |I| < 2^d$
            and both are $1$ otherwise.

            Suppose that $d > d'$ and $C_i U = U' \neq U$.  For any
            $H$ a copy of $Q_{d'}$ in $Q_d[U]$, we have that one of the following
            three options holds: $C^+_i(H) = H$, $C^-_i(H) = H$, or $C^-_i(H) =
            C_{i, \emptyset}(C^+_i(H))$.  Hence by induction the following holds.
            \begin{align*}
            \#(Q_{d'} \subseteq U) \leS
            & \#(Q_{d'} \subseteq C^+_i U) + \#(Q_{d'} \subseteq C^-_i U) \; +\\ 
            &\;\min\Big(\#(Q_{d'-1} \subseteq C^+_i U), \#(Q_{d'-1} \subseteq
            C^-_i U)\Big)\\[5pt]
            \leS& \#(Q_{d'} \subseteq C^+_i U') + \#(Q_{d'} \subseteq C^-_i U')  \;+\\ 
            &\;\min\Big(\#(Q_{d'-1} \subseteq C^+_i U'), \#(Q_{d'-1} \subseteq
            C^-_i U')\Big)\\[5pt]
            \eqS &\#(Q_{d'} \subseteq U')
            \end{align*}
            The first inequality follows from the fact that $C^-_i U'$ and $C_{i,
            \emptyset}(C^+_iU')$ are nested. 

            Define a finite sequence $\{U_i : i = 0, \ldots, K\}$ by taking $U_0 =
            U$ and $U_{k+1} = C_i U_k$ for the least $i$ such that
            $C_i U_k \neq U_k$ if such an $i$ exists.  It is easy
            to verify that this sequence cannot be infinite. Denote $W = U_K$.  Then
            $\#(Q_{d'} \subseteq U) \leq \#(Q_{d'} \subseteq W)$ and
            $C_i W = W$ for every $i \in [d]$.  If $W = I$ the proof
            is complete, thus we may assume that $W \neq I$.

            Since $W \neq I$, $W$ is not initial, so there exists $S
            < T$ with $S \notin W$ and $T \in W$.  Since $C^+_i W$
            and $C^-_i W$ are both initial in the binary order, we have
            that $i \in S \triangle T$ for every $i \in [d]$. In other words, $S =
            T^c$, and there is at most one such pair $(S, T)$, so $T$ is the
            successor of $S$ in binary order and is the maximal element of
            $W$. Hence $T = \{d\}$ and $S = [d-1]$. But then $T$ is in at most one
            $Q_{d'}$ in $W$, whilst $S$ is in $\binom{d-1}{d'} \geq 1$ copies of
            $Q_{d'}$ in $W - T + S$. Hence $I = W - T + S$ has at least as many
            $Q_{d'}$ subgraphs as $W$, completing the proof.
        \end{proof}

        The following upper bound on the number of copies of a subcube follows easily.
        \begin{lem}\label{lem:cubes-bound}
            Let $U$ be a subset of $V(Q_d)$ of size $n$. Then \[\#(Q_{d'} \subseteq
            U) \le \frac{n}{2^{d'}}\binom{\log_2n + 1}{d'}\]
        \end{lem}
        \begin{proof}
            By lemma \ref{lem:subcubes}, we may that assume $U$ is initial
            in binary order, so $U$ is contained in a cube of dimension
            $\lceil \log_2 n\rceil$. Hence each vertex is in at most $\binom{\log_2n
            + 1}{d}$ copies of $Q_{d'}$  and each copy of $Q_{d'}$ is counted
            $2^{d'}$ times.
        \end{proof}

        In fact, one can prove a smooth version of the above upper bound. 
        \begin{lem}\label{lem:cubes-smooth-bound}
            Let $U$ be a subset of $V(Q_\infty)$ of size $n$. Then \[\#(Q_{d}
            \subseteq U) \leq \frac{n}{2^{d}}\binom{\log_2n}{d}\]
        \end{lem}
        \begin{proof}
            Let $T_{d,n}=\#(Q_d\subseteq I_n)$, where $I_n$ is initial in
            binary order in $Q_\infty$ with $|I|=n$. We prove that $T_{d,n}\le
            \frac{n}{2^d}\binom{\log n}{d}$ by induction on $d$. It is clear for
            $d=0$ so we assume $d>0$.

            We proceed by induction on the number of non zero digits in the binary
            representation of $n$.  If $n$ is a power of $2$, $I_n$ is a cube of
            dimension $\log n$ and we have $T_{d,n}= \frac{n}{2^d}\binom{\log
            n}{d}$.

            Now suppose that $n$ has $l>1$ non zero digits in the binary
            representation.  Write $n=2^{k_1}+\ldots+2^{k_l}$ where $k_1>\ldots>k_l$
            and let $r=2^{k_1}$ and $m=n-r$.  Then by the definition of
            binary order and by induction we have
            \begin{align*}
            T_{n,d} 
            \eqS & T_{r,d}+T_{m,d}+T_{m,d-1}
            \leS \frac{r}{2^d}\binom{\log r}{d}+ \frac{m}{2^d}\binom{\log m}{d}+ \frac{m}{2^{d-1}}\binom{\log m}{d-1}
            \end{align*}

            It remains to prove the following inequality.
            \begin{align} \label{eqn:basic-ineq}
            \frac{r}{2^d}\binom{\log r}{d}+ \frac{m}{2^d}\binom{\log m}{d}+ \frac{m}{2^{d-1}}\binom{\log m}{d-1} 
            \leS \frac{n}{2^d}\binom{\log n}{d}
            \end{align}

            If $m < 2^ {d - 1}$ the second and third summands are zero, and it is
            easy to check that the required inequality holds.  We assume that $m\ge
            2^{d-1}$.  Writing $r=(1+\alpha)m$ and rearranging Inequality
            (\ref{eqn:basic-ineq}), we need to show that the following expression is
            non-negative.
            \addspace{
            \begin{align*}
            &\frac{(2+\alpha)m}{2^d}\binom{\log((2+\alpha)m) }{d}- \frac{(1+\alpha)m}{2^d}\binom{\log ((1+\alpha)m)}{d} \\ 
            & \: \: - \: \frac{m}{2^d}\binom{\log m}{d}- \frac{m}{2^{d-1}}\binom{\log m}{d-1}
            \end{align*}
            }
            Writing $\beta=\log m$, we need to show that the following expression is
            non-negative for $\alpha>0$ and $\beta\ge d-1$.
            \addspace{
            \begin{align*}
            f_{\beta}(\alpha) \eqS  &
            (2+\alpha)\binom{\log(2+\alpha)+\beta }{d}- (1+\alpha)\binom{\log (1+\alpha)+\beta}{d} \\
            & - \binom{\beta}{d}- 2\binom{\beta}{d-1}
            \end{align*}
            }

            Substituting $\alpha=0$ we obtain
            \begin{align*}
            f_{\beta}(0) \eqS 
            2\binom{1+\beta }{d}- \binom{\beta}{d}- \binom{\beta}{d}- 2\binom{\beta}{d-1} \eqS 0
            \end{align*}
            The derivative $f_{\beta}'(\alpha)$ at $\alpha>0$ is
            \addspace{
            \begin{align*}
            &\frac{1}{d!\ln 2}\sum\limits_{i=0}^{d-1}\left(\prod_{0\le j\le d-1, j\neq i}(\log (2+\alpha)+\beta-j)-\prod_{0\le j\le d-1, j\neq i}(\log (1+\alpha)+\beta-j) \right)\\ 
            &\:\: + 2\binom{\log (2+\alpha)+\beta}{d}-\binom{\log (1+\alpha)+\beta}{d}
            \end{align*} 
            }
            We conclude that $f_{\beta}(\alpha) \ge 0$ for $\alpha >0$ and
            $\beta\ge d-1$, as required.
        \end{proof}

\section{The star is best for $n=d$} \label{sec:star}
    In this section we prove Theorem \ref{thm:star}, showing that the star
    maximises the largest eigenvalue among all subgraphs of the cube $Q_d$ with at
    most $d$ vertices. 

    \thmStar*
    
    Note that this result is not entirely obvious. Indeed, a
    natural line of attack is to use the inequality $\lambda_1(G)\le
    \sqrt{s(G)}$ of Favaron, Mah\'eo and Sacl\'e~\cite{FMS2} that we mentioned
    in the introduction, where $s(G)$ is the maximum of the sum of degrees of
    vertices adjacent to some vertex. Taking a vertex $u$, its $k$ neighbours, and
    $\binom{k}{2}$ additional vertices, each joined to two of the $k$ neighbours
    of $u$, we get a subgraph $G$ of $Q_n$ with $n=1+k+\binom{k}{2}=
    (k^2+k+2)/2$ vertices and $e(G)=s(G)=k^2$. Hence, $\lambda_1(G)\le
    \sqrt{s(G)}=k$, which is about $\sqrt{2}$ times as large as $\sqrt{n-1}$,
    the bound we wish to prove. The problem is, of course, that the inequality
    we have applied is far from sharp in this case.

    We shall use the following bound, relating the problem of maximising
    the largest eigenvalue to the task of maximising a trace of a matrix.
    \begin{lem}\label{lem:trace-bound}
        Let $G$ be a bipartite graph with bipartition $\{X, Y\}$. Then 
        \begin{align*}
            (\lambda_1(G))^{2k}
            \leS & \frac{1}{2}\tr\left(A(G)^{2k}\right) \\
            \eqS & \#(\text{closed walks of length 2k starting from a vertex in }X) \\
            \eqS &\#(\text{closed walks of length 2k starting from a vertex in }Y).
        \end{align*}
        In particular, 
        $$(\lambda_1(G))^4\le \#(\text{edges in }G)+2\#(\text{paths of length
        }2\text{ in }G)+ 4\#(C_4\text{ in }G).$$
    \end{lem}
    \begin{proof}
        Immediate from the fact that $\left(A^k\right)_{i,j}$ is the number of walks of
        length $k$ from vertex $i$ to vertex $j$.
    \end{proof}

    We shall also make use of the following bound on the number of edges
    and $4$-cycles in a $K_{2, 3}$-free bipartite graph.
    \begin{claim}\label{claim:bound-cycles-by-l}
        Let $G$ be a bipartite graph with bipartition $\{X, Y\}$ and assume that
        $G$ is $K_{2,3}$-free.  Set $k=|X|$, $l=|Y|$. Then
        \begin{itemize}
        \item
        $\#(C_4 \text{ in } G) \le \binom{l}{2}$.
        \item
        $|E(G)|\le \#(2\text{-paths with both ends in }Y)+k\le 2\binom{l}{2}+k$.
        \end{itemize}
    \end{claim}
    \begin{proof}
        The first part follows directly from the fact that $G$ is
        $K_{2,3}$-free, so every pair of vertices in $Y$ is contained in at most
        one $4$-cycle.  The first inequality in the second part follows from the
        observation that for every vertex $v\in X$, we have $d(v)\le
        \#(2\text{-paths in } G\text{ with } v \text{ as the middle vertex})+1$.
        The second inequality again follows from the assumption that $G$ is
        $K_{2,3}$-free.
    \end{proof}

    We now proceed to the proof of Theorem \ref{thm:star}.
    \begin{proof}[ of theorem \ref{thm:star}]
        Let $G$ be a subgraph of $Q_d$ with  $n\le d$ vertices and assume
        $\lambda_1(G)\ge\sqrt{n-1}$.  Denote by $\{X, Y\}$ the bipartition of
        the vertices of $G$ where $k=|X| \ge |Y| = l$.  For a vertex $v\in V(G)$
        denote by $d(v)$ the degree of $v$.

        By Lemma \ref{lem:trace-bound} and  the fact that $G$ does not have
        $K_{2,3}$-free, we obtain the following. 
        \addspace{
        \begin{align*}
        (n-1)^2 \leS & (\lambda_1(G))^4 \\
        \leS & 2\sum\limits_{v\in V}\binom{d(v)}{2}+|E(G)|+4\#(C_4\text{ in }G) \\
        \leS &2\left(\,\binom{k}{2}+\binom{l}{2}+2\#(C_4\text{ in }G )\,\right)+|E(G)|+4\#(C_4\text{ in }G)\\
        \eqS &2l^2-2nl+n^2-n+|E(G)|+8\#(C_4\text{ in }G).
        \end{align*}
        }
        Hence, 
        \begin{equation}\label{eqn_star_main}
        0\leS 2l^2-2nl+n-1+|E(G)|+8\#(C_4\text{ in }G).
        \end{equation}

        We replace $|E(G)|$ and $\#(C_4\text{ in }G)$ by the upper bounds from
        Lemma \ref{lem:cubes-smooth-bound} to obtain the following inequality.
        \begin{align*}
        0\leS &\,2l^2-2nl+n-1+\frac{1}{2}n\log n+2n\binom{\log n}{2}\\
        \eqS &2l^2-2nl+n(\log^2n-\frac{1}{2}\log n+1)-1.
        \end{align*}

        Since $l\le n/2$, we deduce the following upper bound on $l$.
        \begin{align}\label{eqn:upper-bound-l}
        \begin{split}
        l \leS & \frac{1}{4}\left(2n-\sqrt{4n^2-8n(\log^2 n-\frac{1}{2}\log n+1)+8}\,\right) \\
        \eqS & \frac{1}{2}\left(n-\sqrt{n^2-2n\log^2 n+n\log n-2n+2}\,\right).
        \end{split}
        \end{align}

        By Claim \ref{claim:bound-cycles-by-l} and Inequality (\ref{eqn_star_main}), 
        \begin{align*}
        0 \leS &  2l^2-2nl+n-1+10\binom{l}{2}+n-l \\
        \eqS & (l-1)(7l+1-2n).
        \end{align*}

        If  $l\ge 2$  it follows that $l\ge\frac{1}{7}(2n-1)$.  Combining this
        lower bound on $l$ with the upper bound (\ref{eqn:upper-bound-l}), we
        get the following inequality.
        \begin{equation}\label{eqn_star_main2}
        \frac{1}{7}(2n-1)\leS l\leS \frac{1}{2}\left(n-\sqrt{n^2-2n\log^2 n+n\log n-2n+2}\,\right).
        \end{equation}
        This is a contradiction if $n\ge 103$.
        Thus if $n \ge 103$ we must have $l = 1$, implying that $G$ is a star.
    \end{proof}

\section{The largest eigenvalue of the Hamming ball} \label{sec:calc-eval}

    In this section we estimate the largest eigenvalue of the Hamming ball
    $H_d^i$ for several ranges of $i$ and $d$.  We start by proving Theorem
    \ref{thm:hamming-eval}, where we estimate the eigenvalue of the Hamming ball
    when the radius goes to infinity.

    \thmHammingEval*

    The upper bound from Theorem \ref{thm:hamming-eval} follows trivially from
    the following claim. We shall use this claim in subsequent sections,
    therefore we state it here.

    \begin{claim}\label{claim:trivial-bound-ev}
        Let $G$ be a subgraph of $Q_\infty$. Assume that $G$ has maximum degree
        at most $d$ and that $V(G) \subseteq [d]^{\le t}$, where $t \le d/2$.
        Then $\lambda_1(G)\le 2\sqrt{td}$.
    \end{claim}
    \begin{proof}
        For  $j\ge 0$ let $V_j=\{S\in V(G):|S|=j  \}$.  Let $G_j=G[V_j\cup
        V_{j+1}]$.  The graph $G_j$ is a bipartite graph whose vertices from one
        side have degree at most $j+1$ and the vertices from the other side have
        degree at most $d$. Thus $\lambda_1(G_j)\le \sqrt{(j+1)d}$.

        Let $v=(v_S)_{S\in V(G)}$ be an eigenvector with norm $1$ and eigenvalue
        $\lambda_1(G)$.  Define $\alpha_j^2=\sum\limits_{S\in V_j}v_s^2$.  Note
        that $E(G)=\bigcup_{0\le j<t-1}E(G_j)$, hence
        \addspace{
            \begin{align*}
                \lambda_1(G) = \langle A(G)v,v\rangle 
                \leS &\sum\limits_{j=0}^{t-1}\langle A(G_j)v,v\rangle \\
                \leS
                &\sum\limits_{j=0}^{t-1}\left(\alpha_j^2+\alpha_{j+1}^2\right)\lambda_1(G_j)\\
                \leS
                &\sum\limits_{j=0}^{t-1}\left(\alpha_j^2+\alpha_{j+1}^2\right)\sqrt{(j+1)d}\\
                \leS & 2\sqrt{td}.
            \end{align*}
        }
        It follows that $\lam_1(G) \le 2\sqrt{td}$, completing the proof of
        Claim \ref{claim:trivial-bound-ev}.
    \end{proof}

    We now turn to the proof of Theorem \ref{thm:hamming-eval}
    \begin{proof} [ of Theorem \ref{thm:hamming-eval}]
        Denote $\lambda = \lambda_1(H_d^i)$ and $A = A(H_d^i)$.  We first note
        that since every $S\in H_d^i$ satisfies $|S|\le i$ and the maximum
        degree in $H_d^i$ is $d-i+1$, Claim \ref{claim:trivial-bound-ev} implies
        that $\lambda\le 2\sqrt{i(d+1-i)}$.

        We now obtain a lower bound on $\lambda$.  Define the vector $v\in
        \mathbb{R}^{V(H_d^i)}$ by $v_S =
        \mathbbm{1}_{[d]^i}(S)\binom{d}{i}^{-\frac12}$. Note that $||v|| = 1$.
        For every $k < i$ we have 
        \[\lambda^{2k} \geq \langle A^{2k}v, v\rangle
        \geq \binom{d}{i}\frac{1}{k+1}\binom{2k}{k}
        \left((i-k)(d+1-(i-k))\right)^k\binom{d}{i}^{-1},\]
        as $r(d+1 - r)$ is the
        number of choices for which edges to use to move from a set of  $r$ to $r-1$
        and $r-1$ to $r$ respectively, and $r(d+1 - r)$ is an increasing function
        for $r \leq i \leq \frac{d}{2}$. Hence 
        \addspace{
            \[\lambda^{2k} \geq
            k^{-\frac{3}{2}}2^{2k}((i-k)(d+1-(i-k)))^k(1+o(1)).\]
        }
        Thus
        \begin{align*}
         \lambda \geS &
         k^{-\frac{3}{2k}}2\sqrt{(i-k)(d+1-(i-k))}\left(1+o\left(k^{-1}\right)\right) \\
          \eqS & 2\sqrt{(i-k)(d+1-(i-k))}\left(1+O\left(k^{-1}\log
          k\right)\right).
        \end{align*}
        Taking $k = \sqrt{i\log i} \rightarrow \infty$, we get $\lambda \ge
        2\sqrt{i(d+1-i)}\left(1+O\left(i^{-\frac{1}{2}}\log^\frac{1}{2}i\right)\right)$,
        completing
        the proof of Theorem \ref{thm:hamming-eval}.
    \end{proof}

    We now consider the case where the radius of the Hamming ball is fixed.
    \begin{lem}\label{lem:calc-evals-ham-balls}
        There exist constants $\lam_1<\lam_2<\ldots$ such that
        $\lam_1(H_d^i)=\lam_i\sqrt{d}(1+O(1/d))$.
    \end{lem}
    \begin{proof}
        Let $A_i$ be the $(i+1)\times(i+1)$-matrix defined by
        \begin{align*}
        A_{j,k}=
        \left\{
        \begin{array}{l l}
        1 & j=k+1\\
        j & j=k-1\\
        0 &\text{otherwise}\\
        \end{array}
        \right.
        \end{align*} 
        Denote $\lam_i=\lam_1(A_i)$.  Since $A_i$ is a submatrix of $A_{i+1}$,
        and by monotonicity of the largest eigenvalue for matrices with non negative
        entries, we have $\lam_i<\lam_{i+1}$ for every $i$.  In order to
        complete the proof, it suffices to show that
        $\lam(H_d^i)=\lam_i\sqrt{d}(1+O(1/d))$.

        By symmetry, the eigenvector of $A(H_d^i)$ with eigenvalue
        $\lam_1(H_d^i)$ is uniform on $[d]^{(j)}$ for every $0\le j\le i$.
        Denote by $x_j$ the weight of $[j]$ in the eigenvector.  The
        following holds.
        \begin{align*}
        \lam_1(H_d^i) \; x_j=
        \left\{
        \begin{array}{l l}
        dx_1 & j=0\\
        jx_{j-1}+(d-j)x_{j+1} & 0<j<i\\
        ix_{i-1} &j=i\\
        \end{array}
        \right.
        \end{align*}
        Letting $\mu=\lam_1(H_d^i)/\sqrt{d}$ and $y_j=x_j d^{j/2}$, we obtain
        \begin{align*}
        \mu y_j=
        \left\{
        \begin{array}{l l}
        y_1 & i=0\\
        jy_{j-1}+(1+O(1/d))y_{j+1} & 0<j<i\\
        iy_{i-1} &j=i\\
        \end{array}
        \right.
        \end{align*}
        Recalling the definition of $A_i$, this implies that $\mu
        y_j=(A_iy)_j+O(1/d)$, where $y=(y_0,\ldots,y_i)^T$.  It follows (e.g.~by
        looking at the characteristic polynomials) that $|\mu-\lam_i| =O(1/d)$.
        Lemma \ref{lem:calc-evals-ham-balls} follows.
    \end{proof}

\section{Hamming ball is asymptotically best for $i=o(d)$}\label{sec:asym}
    In this section we prove Theorem \ref{thm:hamming-best-asym} showing that
    for $i=o(d)$ the Hamming ball $H^i_d$ asymptotically maximises the $\lam_1$
    among subgraphs of $Q_d$ with the same number of vertices.  Since our proof
    is rather technical, we start with the special case $i = 1$.

    \subsection{Proof of Theorem \ref{thm:hamming-best-asym} for $i = 1$}
        Let us first state the result for the special case $i = 1$.
        \begin{lem}\label{lem:hamm-asy-best-rad-one}
            Let $c>0$ be fixed and let $G$ be a subgraph of $Q_d$ with $n\le cd$
            vertices.  Then $\lambda_1(G)\le \sqrt{d}+O\left(d^{1/4}(\log
            d)^{1/2}\right)$.
        \end{lem}
        Using our results about compressions, we may assume that $V(G)$ is
        compressed. This enables us to partition $V(G)$ into stars, in such a
        way that the edges not covered by the stars have a small contribution to
        the eigenvalue, thus enabling us to obtain the required estimate of
        $\lam_1(G)$.
        \begin{proof}
            By Lemmas \ref{thm:down-closed} and \ref{lem:ij-compressed} we can
            assume that $V(G)$ is compressed. Namely for every $1\le j<k\le d$
            we have $C_{k,\emptyset}(V(G))=V(G)$ and  $C_{k, j}(V(G))=V(G)$.

            We aim to partition $V(G)$ in such a way that each part induces a
            star and that the graph spanned by the edges not contained in any of
            these parts is of small maximal degree.  This would imply that
            $\lambda_1(G)$ is at most the eigenvalue of the star with $d+1$
            vertices plus an error term which can be controlled by the maximal
            degree of the `leftover' edges.

            Let $\eps=\eps(d)=\sqrt{2c/d}$.  Let $\A$ be the set of vertices of
            degree at least $\eps d$ in $G$. We call these vertices `heavy'.  To
            minimise the maximal degree of the leftover graph, we wish to have
            each heavy vertex as a centre of one of the stars in the partition.
            It may happen e.g.~that $\{1\},\{2\}$ are heavy and $\{1,2\}$ is
            not, in which case $\{1,2\}$ will have to appear in two stars of the
            partition. To avoid this from happening, we add vertices to the set
            of heavy vertices as follows.

            Let $\B=\{t\in[d]:\{t\}\in \A\}$. Note that since $V(G)$ is
            compressed, $\A$ is compressed as well, and thus $\B$ is an interval
            and $m=\max \B=|\B|$. Finally define $\D=\mathcal{P}([m])\cap V(G)$.
            Since $\A$ is down-compressed, $\A\subseteq \D$.  We note that the
            maximum degree of $\A$ is at most $\eps d$. Indeed, suppose that $v
            \in \A$ has at least $\eps d$ neighbours in $\A$. Denote this set of
            neighbours by $A$. The every vertex in $A$ has at least $\eps d$
            neighbours in $V(G)$. Note that by the structure of $Q_d$, no vertex
            is a neighbour of more than two vertices of $A$.  It follows that
            $|V(G)| > \frac{|A|\eps d}{2} \ge \frac{(\eps d)^2}{2} \ge n$, a
            contradiction.  In particular, $m=\deg_{\A}(\emptyset)\le \eps d$.

            For  $S\in \D$  define $N^*(S)=\{S\}\cup \big(N(S) \setminus
            \D\big)$, where $N(S)$ denotes the neighbourhood of $S$ in $G$.  We
            claim that $N^*(S)_{S \in \D}$ is a collection of disjoint sets.
            Indeed, by the choice of $\D$, the $\D \subseteq \mathcal{P}([m])$
            and any vertex in the neighbourhood of $D$ which is not in $\D$ must
            be of the form $S \cup \{s\}$ where $s \notin [m]$ and $S \in \D$.
            Furthermore, clearly, each of these sets induces a star.

            Let $v=(v_S)_{S\in V(G)}$ be a vector of norm $1$ with positive
            entries such that $A(G)v=\lambda_1(G)v$.  Note that the edges of $G$
            are covered by the edges of the graphs $G[\D]$, $G \setminus \D$ and
            $\{N^*(S)\}_{S\in \D}$. We thus obtain the following upper bound on
            $\lambda_1(G)$.
            \addspace{
            \begin{align*}
            \lambda_1(G) \eqS & \langle A(G)v,v\rangle \\
            \leS &  \Big(\sum\limits_{S\in \D}\langle A(G[N^*(S)])v,v\rangle\Big )+\langle A(G[\D])v,v\rangle+
            \langle A(G \setminus \D)v,v\rangle\\
            \leS &\Big(\sum\limits_{S\in \D}\lambda_1(G[N^*(S)])\!\!\sum\limits_{T\in N^*(S)}\!\!\!\!v_T^2\Big)+\Big(\lambda_1(G[\D])\sum\limits_{S\in \D}v_S^2\Big)+\Big(\lambda_1(G \setminus \D)\sum\limits_{S\notin \D}v_S^2\Big).
            \end{align*}
            }

            It remains to obtain upper bounds on the largest eigenvalue of the
            graphs $G[\D]$, $G \setminus \D$ and $\{N^*(S)\}_{S \in \D}$.
            Recall that by Claim \ref{claim:trivial-bound-ev}, given a subgraph
            of $Q_d$, we have $\lam_1(G) \le 2\sqrt{\Delta(G) t}$, where
            $\Delta(G)$ is the maximum degree of $G$ and $t$ is the size of the
            largest set in $V(G)$.  Since $\D \subseteq \mathcal{P}([m])$, the
            maximum degree of $G[\D]$ is bounded by $\eps d$.  Also, by
            definition of $\A$, the maximum degree of $G \setminus \D$ is at
            most $\eps d$.  Since $V(G)$ is compressed, the largest set in
            $V(G)$ is of size at most $\log n$.  It follows from Claim
            \ref{claim:trivial-bound-ev} that 
            $$\lambda_1(G[\D]),\,\lambda_1(G
            \setminus \D)\le 2\sqrt{\eps d \log n}.$$

            Furthermore, $\lambda_1(N^*(S))\le \sqrt{d}$ since each set
            $N^*(S)$ is a star with at most $d+1$ vertices.  Thus, by the above
            inequality and using the disjointness of the sets $N^*(S)$, we
            obtain
            \begin{equation*}
                \lambda_1(G)\le 
                \sqrt{d} +O(\sqrt{\eps d \log d})=\sqrt{d} +O(d^{1/4}(\log d)^{1/2}),
            \end{equation*}
            completing the proof of Lemma \ref{lem:hamm-asy-best-rad-one}.
        \end{proof}

    \subsection{Proof of Theorem \ref{thm:hamming-best-asym}}
        We now prove Theorem \ref{thm:hamming-best-asym} in general.
        \thmHammingBestAsymptotically*
        The idea is similar to the special case of $i = 1$. Again we find a
        partition of $G$ with sets whose largest eigenvalue can be bounded by the
        largest eigenvalue of a suitable Hamming ball. Furthermore, we ensure that
        the leftover edges have a small contribution to the largest eigenvalue of
        $G$.

        \begin{proof} [ of Theorem \ref{thm:hamming-best-asym}]
            By Lemmas \ref{thm:down-closed} and \ref{lem:ij-compressed} we can
            assume that $G$ is compressed.  Similarly to the proof for $i=1$, we
            partition the vertices into sets that induce subsets of the Hamming
            ball of radius approximately $i$.  We choose the  partition in such
            a way that the edges not covered by one of these subsets span a
            graph with small maximal degree.  In this way we can bound the
            eigenvalue of the both subgraphs of $G$ to obtain the required
            bound.  In order to define the partition we need some notation.

            Let $\eps=\eps(d)<1$ and define the following sets recursively.
            \begin{flalign*}
            &\qquad \A_0=V(G).&\\
            &\qquad \A_k=\{S\in \A_{k-1}:S\text{ has at least }\epsilon d\text{
            neighbours in }\A_{k-1} \}.&
            \end{flalign*}

            Let $M=\max\{k:\A_k\neq\emptyset\}$.  Note that since $G$ is
            compressed, the sets $(\A_k)_{0\le k\le M}$ are compressed.  The
            sets $\A_k$ measure how `heavy' a vertex is: for a vertex $v\in
            V(G)$, the larger $\max\{k : v\in\A_k\}$ is, the heavier $v$ is.

            As in the proof of the special case, we want to take the heaviest
            vertices to be the centres of the Hamming balls defining the
            partition.  Since we now have many levels, we first take Hamming
            balls centred at the heaviest vertices, then take as centres the
            heaviest vertices among those that weren't covered in the first
            round, and so on.  This process is somewhat complicated by the fact
            that we want each vertex to appear in at most one such Hamming ball.
            To ensure this, we add some of the vertices to sets of heavy
            vertices using the following definitions.

            We define sets $\B_k$, $\C_k$, $\D_k$, $\E_k$ and numbers $m_k$ for
            $0 \le k \le M$ as follows.  For $k=0$,
            \begin{flalign*}
            &\qquad\mathcal{B}_0=\{t\in [d]: \{t\}\in \A_M \}\cup\{1\}&\\
            &\qquad m_0=\max\mathcal{B}_0&\\
            &\qquad\C_0=\emptyset&\\
            &\qquad\E_0=\D_0=\mathcal{P}([m_0])\cap V(G)&
            \end{flalign*}

            For $0<k\le M$ define recursively
            \begin{flalign*}
            &\qquad\mathcal{B}_k=\{t>m_{k-1}+1: \{m_0+1,\ldots,m_{k-1}+1,t\}\in
            \mathcal{A}_{M-k}\}\cup \{m_{k-1} + 1\}& \\
            &\qquad m_k=\max\B_k&\\
            &\qquad\mathcal{C}_k=\{S\cup \{t\}: S\in \C_{k-1}\cup \D_{k-1}, t>m_{k-1}\}& \\
            &\qquad\mathcal{D}_k=(\mathcal{P}([m_k]) \cap V(G))\setminus(\E_{k-1}\cup \C_k)&\\
            &\qquad\mathcal{E}_k=\C_0\cup\ldots\cup\C_k\cup\D_0\cup\ldots\cup \D_k& 
            \end{flalign*}

            Before we proceed with the proof, we try to convey the ideas behind
            the above definitions.  The sets $\D_k$ defined above will be the
            centres of the Hamming balls and the $\C_k$'s will consist of the
            other vertices covered by these balls.  In each stage we define
            $\C_k$ to be the set of neighbours of vertices which appeared
            previously.  We define $\D_k$ so as to be the up-closure (relatively
            to $V(G)$) of the vertices in $\A_{M-k}$ which were not covered
            previously.  To this end, in each stage $B_k$ and $m_k$ are defined
            so that every $t\in S\in \A_{M-k} \setminus \big(\E_{k-1}\cup
            \C_k\big )$ satisfies $t\le m_k$. Thus $\D_k$ contains $\A_{M-k}
            \setminus \big(\E_{k-1}\cup \C_k\big )$ and is up-closed in $V(G)$.

            We now define the partition of $V(G)$ into sets inducing subgraphs
            of Hamming balls with centres in $\bigcup_{0\le k< M}\D_k$.  For a
            vertex $S\in V(G)$ and $t\ge 1$, let $N_t(S)$ denote the set of
            vertices of $V(G)$ in distance $t$ from $S$.  For every $0\le k<M$
            and every $S\in \D_k$, let 
            \begin{equation*}
            N^{(k)}_S=\{S\}\cup\bigcup_{1\le j\le M-k}\Big(N_{j}(S)\cap \C_{k+j}\Big).
            \end{equation*}

            In order to show that the sets $N_k(S)$ satisfy our requirement we
            need the following proposition.  Its proof is delayed to the end of
            this section.
            \begin{restatable}{prop}{propClarify}\label{prop:clarify} The
                following assertions hold.
            \begin{enumerate}
            \item\label{part:1-partition-NkS}
            The sets $N^{(k)}(S)$, where $0\le k\le M-1$ and $S\in \D_k$, are pairwise disjoint.
            \item\label{part:2-partition-Ck-Dk}
            The sets $\C_k\cup\D_k$, where $0\le k\le M$, form a partition of $V(G)$.
            \item\label{part:3-edges}
            $E(G)= \left( \bigcup_{0\le k\le M-1,\: S \in \D_k}E(G[N^{(k)}(S)]) \right)
            \cup \left(\bigcup_{0\le k\le M}E(G[\C_k\cup \D_k])\right).$
            \item\label{part:4-max-degree}
            The maximum degree of $G[\C_k\cup \D_k]$, where $0 \le k \le M$, is at most $\eps d$.
            \end{enumerate}
            \end{restatable}

            Let $v=(v_S)_{S\in V(G)}$ be a vector of positive weights on the
            vertices of $G$ with norm $1$, satisfying $A(G)v=\lambda_1v$. Define
            \addspace{
            \begin{flalign*}
            &\qquad\alpha_k^2=\sum\limits_{S\in\C_k\cup D_k}v_S^2\quad\text{ for
            }0\le k\le M.&\\
            &\qquad(\beta_{k,S})^2=\sum\limits_{T\in
                N^{(k)}_S}v_{T}^2\quad\text{ for }0\le k< M \text{ and }S\in
                \D_k.&
            \end{flalign*}
            }
            By Parts (\ref{part:1-partition-NkS}) and
            (\ref{part:2-partition-Ck-Dk}) above,
            $\sum\limits_{0\le k<M}\sum\limits_{S\in
            \D_k}(\beta_{k,S})^2\le 1$ and $\sum\limits_{k=0}^{M}\alpha_k^2=1$.
            Thus, by Part (\ref{part:3-edges}),
            \begin{align}\label{eqn:ev-estimate}
            \begin{split}
                \lambda_1(G) \eqS & 
                \langle A(G)v,v\rangle \\
                \leS & \sum\limits_{k=0}^{M}\Big\langle A(G[\C_k\cup \D_k])v,\,v\Big\rangle+
                \sum\limits_{k=0}^{M}\sum\limits_{S\in \D_k}\Big\langle
                A(G[N^{(k)}_S])v,v\Big\rangle \\
                \leS &\sum\limits_k\alpha_k^2\cdot\lambda_1(G[\C_k\cup \D_k])+
                \sum\limits_{k,S}(\beta_{k,S})^2\cdot\lambda_1(G[N^{(k)}_S])\\
                \leS &\max_{k}\lambda_1(G[\C_k\cup \D_k])+\max_{k,S}\lambda_1(G[N^{(k)}_S])
            \end{split}
            \end{align}

            By Part (\ref{part:4-max-degree}) of Proposition \ref{prop:clarify},
            the maximum degree of $G[\C_k\cup \D_k]$ is at most $\eps d$. Since
            $V(G)$ is compressed, the largest set in $V(G)$ has size at most
            $\log n$.  Recall that $n = \Theta\left(\binom{d}{i}\right)$, thus
            $\log n = (1 + o(1))i \log(d / i)$.  We conclude the following upper
            bound by Claim \ref{claim:trivial-bound-ev}.
            \begin{equation} \label{eqn_upper_bound_Ck_Dk}
            \lambda_1(G[\C_k\cup \D_k])\leS 2\sqrt{\eps d\log n}\eqS 2(1+o(1))\sqrt{\eps di\log (d/i)}.
            \end{equation}

            Let us treat first the case where $i \rightarrow \infty$.  By
            Theorem \ref{thm:hamming-eval}, using the monotonicity of the
            largest eigenvalue of a graph,
            \begin{equation} \label{eqn_upper_bound_NkS}
            \lambda_1(G[N^{(k)}_S])\le \lambda_1(H^{M-k}_{d})\le\lambda_1(H^{M}_{d})=2(1+o(1))\sqrt{M(d-M)}.
            \end{equation}
            Substituting Inequalities (\ref{eqn_upper_bound_Ck_Dk}) and
            (\ref{eqn_upper_bound_NkS}) into the Inequality
            (\ref{eqn:ev-estimate}), it follows that
            \begin{equation}\label{eqn_eignevalue}
            \lambda_1(G)\le 2(1+o(1))\Big(\sqrt{\eps di\log (d/i)}+\sqrt{M(d-M)}\Big).
            \end{equation}

            The following claim will imply that we can choose $\eps$ so as to
            make the above upper bound arbitrarily close to $\lambda_1(H^i_d)$.
            \begin{claim} \label{claim:bound-on-depth}
                Let $\alpha>0$ and set $\eps=\frac{\alpha}{\log(d/i)}$.
                Then $M\le (1+o(1))i$.
            \end{claim}
            \begin{proof}
                For arbitrary $\beta >0$ we show that $M\le (1+\beta)i$ for large
                enough $d$.  Let $N=(1+\beta)i$ and $D=\eps d$.  We need to show
                that $\A_N=\emptyset$.  Assuming the contrary, let $S\in \A_N$.
                Then $S$ has at least $D$ neighbours in $\A_{N-1}$, which in turn
                have at least $\binom{D}{2}$ new neighbours in $\A_{N-2}$ and so on.
                It follows that $n=|V(G)|\ge
                1+D+\binom{D}{2}+\ldots+\binom{D}{N}=\left|[D]^{(\le N)}\right|$.
                Recall that $i = o(d)$ and note that $\frac{N}{D} = \frac{1 +
                \beta}{\alpha}\cdot \frac{\log(d/i)}{d/i} = o(1)$, i.e.~$N = o(D)$.
                Thus,
                \begin{align*}
                \left|[D]^{(\le N)}\right|=(1+o(1))\frac{1}{\sqrt{2\pi N}}\left(\frac{eD}{N}\right)^N.
                \end{align*}
                On the other hand,
                \begin{align*}
                n\le c\left|[d]^{(\le i)}\right|=(1+o(1))\frac{c}{\sqrt{2\pi i}}\left(\frac{ed}{i}\right)^i.
                \end{align*}
                Combining the two inequalities, we obtain the following.
                \addspace{
                \begin{align*}
                 \frac{c}{\sqrt{i}} \left(\frac{ed}{i}\right)^i
                \geS & (1+o(1))\frac{1}{\sqrt{N}}\left(\frac{eD}{N}\right)^N \\
                \eqS & (1+o(1))\frac{1}{\sqrt{(1+\beta)i}}\left(\frac{e\alpha}{1+\beta}\cdot\frac{d/i}{\log(d/i)}\right)^{(1+\beta)i}.
                \end{align*}
                }
                We obtain the following inequality, where $c_1, c_2$ are
                constants depending on $\alpha, \beta, c$.
                \begin{align*}
                c_2\ge \Big(c_1\frac{(d/i)^{\frac{\beta}{1+\beta}}}{\log(d/i)}\Big)^{(1+\beta)i}
                \end{align*}
                Since $i=o(d)$, we have $\log(d/i)=o\left((d/i)^\gamma\right)$
                for every fixed $\gamma>0$ and we have reached a contradiction.
                This implies that $M\le (1+\beta)i$ for large $d$.
            \end{proof}

            By Inequality (\ref{eqn_eignevalue}) with
            $\eps=\frac{\alpha}{\log(d/i)}$ we have
            \begin{align*}
            \lambda_1(G)\leS & 2(1+o(1))(\sqrt{\alpha i d}+\sqrt{i(d-i)}) \\ \eqS & 2(1+\sqrt{\alpha})(1+o(1))\sqrt{i(d-i)}.
            \end{align*}

            Since $\alpha$ can be taken arbitrarily close to $0$, it follows that
            \begin{equation*}
            \lambda_1(G)\le 2(1+o(1))\sqrt{i(d-i)}=(1+o(1))\lambda_1(H^{i}_{d}).
            \end{equation*}
            This completes the proof of Theorem
            \ref{thm:hamming-best-asym} in case $i=\omega(1)$.

            It remains to consider the case where $i$ is constant.  Take $\eps =
            2id^{-1/(i + 1)}$. It is easy to check that $\binom{\eps d}{i + 1} >
            n$, implying that $M \le i$ similarly to the proof of Claim
            \ref{claim:bound-on-depth}.  It follows from Inequalities
            (\ref{eqn:ev-estimate}), (\ref{eqn_upper_bound_Ck_Dk}) and Claim
            \ref{claim:trivial-bound-ev} that 
            \begin{equation*}
            \lambda_1(G)\leS 
            O\left(\sqrt{d^{1 - \frac{1}{i + 1}} \log d} \; \right)+\lambda_1(H^i_d) \eqS 
            (1 + o(1))\lambda_1(H^i_d),
            \end{equation*}
            completing the proof of Theorem \ref{thm:hamming-best-asym}.
        \end{proof}

    \subsection{Proof of Proposition \ref{prop:clarify}}
    In order to complete the proof of Theorem
    \ref{thm:hamming-best-asym}, it remains to prove Proposition
    \ref{prop:clarify}. 
    \propClarify*

    \begin{proof}
        Recall the definition of the sets $\A_k, \B_k, \C_k, \D_k, \E_k$.
        \begin{flalign*}
        &\qquad \A_0=V(G).&\\
        &\qquad \A_k=\{S\in \A_{k-1}:S\text{ has at least }\epsilon d\text{
        neighbours in }\A_{k-1} \}.&
        \end{flalign*}

        For $k=0$,
        \begin{flalign*}
        &\qquad\mathcal{B}_0=\{t\in [d]: \{t\}\in \A_M \}\cup\{1\}&\\
        &\qquad m_0=\max\mathcal{B}_0&\\
        &\qquad\C_0=\emptyset&\\
        &\qquad\E_0=\D_0=\mathcal{P}([m_0])\cap V(G)&
        \end{flalign*}

        For $0<k\le M$,
        \begin{flalign*}
        &\qquad\mathcal{B}_k=\{t>m_{k-1}+1: \{m_0+1,\ldots,m_{k-1}+1,t\}\in
        \mathcal{A}_{M-k}\}\cup \{m_{k-1} + 1\}& \\ &\qquad m_k=\max\B_k&\\
        &\qquad\mathcal{C}_k=\{S\cup \{t\}: S\in \C_{k-1}\cup \D_{k-1}, t>m_{k-1}\}& \\
        &\qquad\mathcal{D}_k=(\mathcal{P}([m_k]) \cap V(G))\setminus(\E_{k-1}\cup \C_k)&\\
        &\qquad\mathcal{E}_k=\C_0\cup\ldots\cup\C_k\cup\D_0\cup\ldots\cup \D_k& 
        \end{flalign*}

        We prove the following assertions.
        \begin{enumerate}
        \item\label{item:1-unique-rep}
        For every $S\in \C_k\cup \D_k$ there is a unique $j\le k$ and a unique $T\in D_j$, such that there exist distinct $t_{j+1},\ldots,t_k$ satisfying $t_l>m_{l-1}$ and $S=T\cup\{t_{j+1},\ldots,t_k\}$.
        \item\label{item:2-compressed}
        $\E_k$ is compressed.
        \item\label{item:3-edges}
        There are no edges of $G$ between $\E_k$ and $\D_{k+1}\cup\C_{k+2}\cup \D_{k+2}$.
        \item\label{item:4-partition-Ck-Dk}
        $(\C_k\cup \D_k)\cap \E_{k-1}=\emptyset$.
        \item\label{item:5-max-degree}
        The maximum degree of $G[\C_k\cup \D_k]$ is at most $m_k\le \eps d$.
        \item\label{item:6-everything-covered}
        $\A_{M-k}\subseteq\E_k$.
        In particular, $\E_M=V(G)$.
        \item\label{item:7-partition-NkS}
        The sets $(N^{(k)}(S))_{0\le k\le M-1, S\in \D_k}$ are pairwise disjoint.
        \end{enumerate}

        Note that Proposition \ref{prop:clarify} follows from these assertions.
        Indeed, Part (\ref{part:2-partition-Ck-Dk}) follows from Assertions
        (\ref{item:4-partition-Ck-Dk}) and (\ref{item:6-everything-covered}),
        Part (\ref{part:3-edges}) follows from Assertions (\ref{item:3-edges})
        and (\ref{item:6-everything-covered}) and Parts
        (\ref{part:1-partition-NkS}) and (\ref{part:4-max-degree}) are among
        these assertions.

        \paragraph*{Proof of Assertion (\ref{item:1-unique-rep}).}
            We prove Assertion (\ref{item:1-unique-rep}) by induction on $k$.  It is
            trivial for $k=0$, so we assume $k>0$.  Let $S\in \C_k\cup\D_k$.  Assume
            that $S=T\cup\{t_{j+1},\ldots,t_k\}=R\cup\{r_{l+1},\ldots,r_k\}$, where
            $T\in \D_j$, $R\in \D_l$ and $t_u, r_u>m_{u-1}$ for all $u$.  We show
            that we must have $j = l$ and $T = R$.

            Note that if $j=l=k$, there is nothing to prove.  If $j<k$ it follows
            from the definitions that $S\in \C_k$, thus $S\notin\D_k$ and so $l<k$.
            By the definitions, there is $s\in S$ with $s>m_{k-1}$ such that $S
            \setminus \{s\}\in \C_{k-1}\cup \D_{k-1}$.  Since $T\subseteq [m_j]$ and
            $R\subseteq [m_l]$ it follows that $s\in
            \{t_{j+1},\ldots,t_k\}\cap\{r_{l+1},\ldots,r_k\}$.  Without loss of
            generality, $s=r_k=t_k$.  It follows  that
            $T\cup\{t_{j+1},\ldots,t_{k-1}\}=R\cup\{r_{l+1},\ldots,r_{k-1}\}\in
            \C_{k-1}\cup \D_{k-1}$.  By induction, $j=l$ and $R=T$.

        \paragraph*{Proof of Assertion (\ref{item:2-compressed}).}
            Again we prove the assertion by induction on $k$.  For $k=0$ it
            follows from the definition of $\C_0,\D_0$ and the assumption that
            $G$ is compressed.  Let $k>0$, $S\in \C_k\cup\D_k$ and choose $a\in
            S$, $b<a$ such that $b\notin S$ (if such $b$ exists).  Let $T=S
            \setminus \{a\}$ and $R=S\triangle\{a,b\}$.  To prove the assertion
            we show that $R,T\in \E_k$.

            If $S\in \D_k$, the claim follows directly from the definition of
            $\D_k$ and the fact that $G$ is compressed.  Thus we assume $S\in
            \C_k$, so we can write $S=S_1\cup\{s\}$ where $s>m_{k-1}$ and $S_1\in
            \C_{k-1}\cup \D_{k-1}$.  If $a\neq s$, by induction we have
            $R \setminus \{s\}, T\setminus \{s\}\in \E_{k-1}$, thus $R,T\in
            \E_k$, so we assume $a=s$.  Then clearly $T=S_1\in \C_{k-1}\cup
            \D_{k-1}$. It remains to show that $R\in \E_k$.

            Let $S=S_2\cup\{s_{j+1},\ldots,s_k\}$, be a representation of $S$ as
            in Assertion (\ref{item:1-unique-rep}) and assume $s=s_k$.  If $b\le
            m_j$, it follows that $S_2\cup\{b\}\in \E_j$ and thus $S=(S_2\cup
            \{b\})\cup \{s_{j+1},\ldots,s_{k-1}\}\in \E_{k-1}$.  It remains to
            consider the case $b>m_j$.  Let $s_{j+1}'<\ldots<s_k'$ be such that
            $\{ s_{j+1}',\ldots,s_k' \}=\{ s_{j+1},\ldots,s_{k-1},b \}$.  If
            $s_u'>m_{u-1}$ for every $j+1\le u\le k$ then $R\in \C_k$.
            Otherwise let $l=\max\{u:s_u'\le m_{u-1}\}$ and $S_3=S_2\cup
            \{s_{j+1}',\ldots,s_l'\}$. Since $S_3\subseteq [m_{l-1}]$ it follows
            that $S_3\in \E_{l-1}$ and $R=S_3\cup\{s_{l+1},\ldots,s_k\}\in
            \E_{k-1}$.

        \paragraph*{Proof of Assertion (\ref{item:3-edges}).}
            Let $S\in \E_k$ and $T$ be a neighbour of $S$ in $G$.  We show that
            $T\in \E_k\cup \C_{k+1}$, implying that there are no edges of $G$
            between $\E_k$ and $\D_{k+1}\cup\C_{k+2}\cup \D_{k+2}$.  If
            $T\subseteq S$, it follows from Assertion (\ref{item:2-compressed})
            that $T\in \E_k$.  So we assume $T=S\cup \{t\}$ and set $s=\max S$.
            If $t>m_k$, then $T\in \C_{k+1}$.  If $s,t\le m_k$, then
            $T\subseteq[m_k]$, so $T\in \E_k$.  Finally, we consider the case
            $t<m_k\le s$.  Since $\E_k$ is compressed, it follows that $T
            \setminus \{s\}=S\triangle\{s,t\}\in \E_k$.  This implies $T\in
            \E_k\cup \C_{k+1}$.

        \paragraph*{Proof of Assertion (\ref{item:4-partition-Ck-Dk}).}
            From the definitions it follows that $\E_{k-1}\cap\D_k=\emptyset$.
            Since $\D_0\cup\ldots\cup\D_{k-1}\subseteq\mathcal{P}([m_{k-1}])$,
            it follows $\C_k\cap(\D_0\cup\ldots\cup\D_{k-1})=\emptyset$.  Thus
            it remains to show that $\C_j\cap \C_k=\emptyset$ for  $0\le j<k$.
            We prove this by induction on $k$.  For $k=0$ there is nothing to
            prove.  Assume $0\le j<k$ and $S\in \C_k\cap\C_j$.  Write $s=\max
            S$.  By considering the representations of $S$ as in Assertion
            (\ref{item:1-unique-rep}), it is easy to see that $S\setminus
            \{s\}\in (\C_{j-1}\cup\D_{j-1})\cap(\C_{k-1}\cup\D_{k-1})$.  As
            explained above this implies that $S \setminus \{s\}\in
            \C_{j-1}\cap\C_{k-1}$, contradicting the induction hypothesis.

        \paragraph*{Proof of Assertion (\ref{item:5-max-degree}).}
            Let $S,T\in \C_k\cup\D_k$ and $t\in [d]$ be such that
            $T=S\cup\{t\}$.  If $t>m_k$, it follows from the definitions that
            $T\in \C_{k+1}$, contradicting
            $\C_{k+1}\cap(\C_k\cup\D_k)=\emptyset$.  Thus $t\le m_k$, implying
            that the maximum degree of $G[\C_k\cup\D_k]$ is at most $m_k$.

            We now prove by induction on $k$ that $m_k\le \eps d$.  Recall that
            by the definition of $M$, the maximum degree of $G[\A_M]$ is at most
            $\eps d$.  Thus for $k=0$ we have
            $m_0=\max(1,\deg_{G[\A_M]}(\emptyset))\le \eps d$.  Now let $k>0$
            and $S=\{m_0+1,\ldots,m_{k-1}+1\}$.  It follows from the definition
            of $\B_{k-1}$ that $S\notin \A_{M-(k-1)}$, so
            $\deg_{G[A_{M-k}]}(S)\le \eps d$.  Since $\A_{M-k}$ is compressed,
            $m_k\le\max(\deg_{G[A_{M-k}]}(S), m_{k-1})\le \eps d$.

        \paragraph*{Proof of Assertion (\ref{item:6-everything-covered}).}
            Let $S\in \A_{M-k}$.  Note that if $|S|\le k$ it can be  easily
            shown by induction that $S\in \E_{k}$.  Thus we assume $|S|\ge k+1$.
            Define $t_k=\max S$ and for $0\le j<k$, denote $t_j = \max(S
            \setminus \{t_{j + 1}, \ldots, t_k \})$.  Assume first that
            $m_j<t_j$ for every $0\le j<k$.  Since $\A_{M-k}$ is compressed, it
            follows that $\{m_0+1,\ldots,m_{k-1}+1,t_k\}\in \A_{M-k}$. Thus
            $t_k\le m_k$, $S\subseteq[m_k]$ and $S\in \E_k$.  Otherwise, let
            $l\ge 0$ be maximal such that $t_l\le m_l$. It follows from the
            definitions that $S\cap [m_l]\in \E_l$ and $S\in \E_k$.

        \paragraph*{Proof of Assertion (\ref{item:7-partition-NkS}).}
            We show that for every $k,j$  if $S\in \C_k$, $T\in \D_j$ are such that
            $S\in N_{k-j}(T)$ then there exist $t_{j+1},\ldots,t_k$ such that
            $S=T\cup \{t_{j+1},\ldots,t_k\}$ and $t_l>m_{l-1}$ for every $j<l\le k$.
            This proves that the sets $(N^{(k)}(S))_{0\le k\le M-1, S\in \D_k}$ are
            pairwise disjoint using Assertion (\ref{item:1-unique-rep}).

            By Assertions (\ref{item:2-compressed}) and (\ref{item:3-edges}) the
            sets $\E_l$ are down-closed and there are no edges between $\E_l$ and
            $\E_{l+2}$. Thus, since  $S\in N_{k-j}(T)$, $S$ is obtained by adding
            $k-j$ elements to $T$, and we can write $S=T\cup
            \{t_{j+1},\ldots,t_k\}$.  Assuming that $t_{j+1}<\ldots<t_k$ and that
            there exists $j+1\le l\le k$ such that $t_l\le m_{l-1}$, we define $r$
            to be the maximal such $l$.  It follows that $T\cup
            \{t_{j+1},\ldots,t_r\}\in \E_{r-1}$ and thus $S\in \E_{k-1}$,
            contradicting our assumptions.
    \end{proof}
    The proof of Proposition \ref{prop:clarify} completes the proof of our first
    main result, Theorem \ref{thm:hamming-best-asym}.

\section{Hamming ball is best for fixed $i$}\label{sec:fixed-radius}
    In this section we prove Theorem \ref{thm:fixed-radius}, whose statement is
    as follows.
    \thmFixedRadius*

    Let us start with an outline of the proof.  We are given a graph $G$ that
    maximises the largest eigenvalue among subgraphs of $Q_d$ with $|H_d^i|$
    vertices. As usual, we assume the graph and the eigenvector $v$ with
    eigenvalue $\lam_1(G)$ are compressed. Using the proof of Theorem
    \ref{thm:hamming-best-asym}, we conclude that by removing the
    vertices of level $i+1$ and higher, the largest eigenvalue does not decrease by
    much. We infer that $G$ has to contain almost all vertices levels $i$ or
    less.  By assuming that $G$ maximises $\lam_1$, given an eigenvector, we
    know that moving weight from a vertex of level $i+1$ or higher into level
    $i$ can only decrease the inner product $\langle A(G)v,v \rangle$, enabling
    us to obtain a lower bound on the weight of a vertex at the highest non
    empty level. Finally, using the relations between the weights of vertices
    and their neighbourhoods, and the fact that there are few vertices in level
    $i+1$ or higher, we reach a contradiction to the assumption that $v$ is
    compressed, by concluding that there is a vertex of weight higher than the
    weight of the empty set.

    We now proceed to the proof of the theorem.  
    \begin{proof}[ of Theorem
        \ref{thm:fixed-radius}] 
        Let $G$ be a subgraph of $Q_d$ with $|H_d^i|$
        vertices and assume $\lam\triangleq\lambda_1(G)$ is maximal among
        subgraphs of $Q_d$ with the same number of vertices.  Let $v=(v_S)_{S\in
        V(G)}$ be a positive vector of norm $1$ giving $\lambda=\langle
        A(G)v,v\rangle$.  By Lemmas \ref{thm:down-closed} and
        \ref{lem:ij-compressed}, we can assume that $V(G)$ and $v$ are
        compressed.

        We first show that under the above assumptions, the graph obtained from $G$
        by removing vertices of level $i+1$ or more still has a large maximal
        eigenvalue.

        \begin{claim} \label{claim:big-eval-fst-i-levels} Let $U=V(G)\cap [d]^{(\le
            i)}$.  There exists $\eta=\eta(i)>0$ such that $\lambda_1(Q_d[U])\ge
            \lambda_1(H_d^i)-O(d^{1/2-\eta})$.  
        \end{claim} 
        \begin{proof} We use the
            proof of Theorem \ref{thm:hamming-best-asym}.  Consider
            Inequality (\ref{eqn:ev-estimate}) which states the following.
        \begin{align*}
        \lambda_1(G)\le \max_{k}\lambda_1(G[\C_k\cup \D_k])+\max_{k,S}\lambda_1(G[N^{(k)}_S]).
        \end{align*}

            As explained in the proof of Claim \ref{claim:bound-on-depth}, if $\eps
            = 2id^{-1/(i+1)}$ then $M \le i$, implying that the sets $N_S^{(k)}$ are
            subsets of Hamming balls of radius at most $i$.  It follows that
            $\lam_1(N_S^{(k)})\le \lambda_1(Q_d[U])$, because $G$ is compressed.
            Furthermore, for our choice of $\eps$, we have $$\lam_1(G[\C_k\cup
            \D_k])=O\left(\sqrt{d^{1 - \frac{1}{i + 1}} \log d}\;\right).$$

            Thus for any $\eta < 1/2(i + 1)$ we have
            \begin{align*}
            \lambda_1(G)\le O(d^{1/2-\eta})+\lambda_1(Q_d[U]).
            \end{align*}
            The proof of Claim \ref{claim:big-eval-fst-i-levels} follows from the
            assumption that $\lam_1(G) \ge \lam_1(H_d^i)$.
        \end{proof}

        We conclude that $\left|V(G) \setminus [d]^{(\le i)}\right|$ is small.

        \begin{claim}\label{claim:theta}
            There exists $\theta = \theta(i)>0$ such that $\left|V(G) \setminus
            [d]^{(\le i)}\right|=O(d^{i-\theta})$.
        \end{claim}
        \begin{proof}
            Define 
            \begin{align*}
            & A=\big\{a\in[d]\, : \text{there exists } S\in V(G)\cap [d]^{(i)}\text{ such that } a=\min S\big\}\\
            & B=\big\{S\in [d]^{(i)}\, :  S\cap A\neq \emptyset\big\}.
            \end{align*}
            Since $G$ is compressed, it follows that $A^{(i)}\subseteq V(G)\cap
            [d]^{(i)}\subseteq B$.  Write $|A|=(1-\beta) d$ and let $H$  be the
            subgraph of $Q_d$ induced by $[d]^{(<i)}\cup B$.  Note that $V(G)
            \cap [d]^{(\le i)} \subseteq V(H)$.  It follows from Claim
            \ref{claim:big-eval-fst-i-levels} that $\lam_1(H)\ge
            \lam_1(H_d^i)-O(d^{1/2-\eta})$.  We shall conclude that
            $\beta=O(d^{-\theta})$ for some $\theta=\theta(i)>0$.  This implies
            that $\left| V\cap [d]^{(i)} \right| \ge
            \binom{(1-\beta)d}{i}=\binom{d}{i}-O(d^{i-\theta})$, as required.

            Note that by symmetry, the eigenvector of $H$ with eigenvalue
            $\lam_1(H)$ is uniform on vertices from the same level and with the
            same number of elements in $[(1-\beta)d]$.  Let $x_{j,k}$ be the
            weight of a vertex from $[d]^{(j)}$ with $k$ elements in
            $[(1-\beta)d]$ in the eigenvector. Let $y_{j,k}=x_{j,k}d^{i/2}$ and
            denote $\mu=\lam_1(H)d^{-1/2}$.  Consider the following equation.
            \begin{align*}
            \mu \, y_{j, k} \eqS &
            (j - k)\, y_{j - 1, k} + k \, y_{j - 1, k - 1} \\
            + \: & 
            \mathbbm{1}_{j < i}\cdot 
            \Big( (1 - \beta + O(1 / d )) \,  y_{j + 1, k + 1} 
            + (\beta - O(1 / d)) \, y_{j + 1, k}\Big). 
            \end{align*}

            This system of equations, taken for $0 \le j \le i$ and $0 \le k \le
            j$ describes $H_d^i$, so the corresponding maximal eigenvalue is
            $\lam_1(H_d^i)/\sqrt{d}$.

            We obtain a similar system, by dropping the equation for $j = i$ and
            $k = 0$ (and taking $x_{i, 0} = 0$). Denote the corresponding
            eigenvalue by $\mu_{\beta}$.  By monotonicity of the eigenvalue,
            $\mu_{\beta}$ is decreasing with $\beta$.

            Now consider the systems obtained from the above two systems by
            omitting the $O(1 / d)$ terms.  The eigenvalue from the first system
            is the constant $\lam_i$ from Lemma \ref{lem:calc-evals-ham-balls}.
            Let $\nu_{\beta}$ be the eigenvalue of the second system with the
            $O(1/d)$ terms omitted.  Then $|\mu_{\beta} - \nu_{\beta}|=O(1/d)$.
            Furthermore, by monotonicity of the maximal eigenvalue, $\nu_{\beta} <
            \lam_i$ for $0 < \beta < 1$.

            Note that we can conclude that $\beta = o(1)$.  Suppose to the
            contrary that $\beta > c$ where $c > 0$ is a constant.  Then $\lam_i
            - \nu_c = \Omega(1)$ and $\mu_\beta \le \nu_c + O(1/d)$, implying
            that $\mu_{\beta} = \lam_i - \Omega(1)$, a contradiction.

            We now prove a sharper upper bound on $\beta$.  Denote by
            $P_{\beta}(x)$ the characteristic polynomial obtained by the second
            system with the $O(1/d)$ terms omitted.  Then $P_{\beta}(\lam_i)$ is
            a polynomial in $\beta$ which has a root at $\beta = 0$ but is not
            identically $0$.  Thus $P_{\beta}(\lam_i) = \Omega(\beta ^t)$ where
            $t$ is the smallest power of $\beta$ in $P_{\beta}(\lam_i)$ with a
            non zero coefficient. The degree of this polynomial is at most
            $(1+2+\ldots+i)+i\le 2i^2$, thus
            $P_{\beta}(\lam_i)=\Omega(\beta^{2i^2})$.  Note that the derivative
            of $P_{\beta}(x)$ satisfies $\frac{\partial P_{\beta}(x)}{\partial
            x} = O(1)$ for $|x|\le \lam_i$.  It follows that $P_{\beta}(\lam_i)
            - P_{\beta}(\nu_{\beta}) = O(\lam_i - \nu_{\beta})$.  Since
            $P_{\beta}(\nu_{\beta}) = 0$ and $P_{\beta}(\lam_i) =
            \Omega(\beta^{2i ^ 2})$, this implies that $\lam_i -
            \nu_{\beta}=\Omega(\beta^{2i^2})$.

            If $\beta^{2i^2} =O(1/d)$, we are done. Otherwise, this implies that
            $\frac{\lam_1(H_d^i)-\lam_1(H)}{\sqrt{d}}=\Omega(\beta^{2i^2})$.
            Recall that we have
            $\frac{\lam_1(H_d^i)-\lam_1(H)}{\sqrt{d}}=O(d^{-\eta})$.  Thus
            $\beta=O(d^{-\eta/ 2i^2})$, completing the proof.
        \end{proof}

        Let $l=\max\{j:V(G)\cap [d]^{(j)}\neq\emptyset \}$. Assuming that
        $V(G)\neq [d]^{(\le i)}$, we have $l > i$.  Note that since $G$ is
        down-compressed, we have $l = O(\log d)$.

        \begin{claim} \label{claim:val-of-l}
            If $l>i$ then $v_{[l]}=\Omega(v_\emptyset \lam^{-i})$.
        \end{claim}
        \begin{proof}
            We first show by induction on $|S|$ that for every $S\in V(G)$ we
            have $v_S\ge v_\emptyset\lam^{-|S|}$.  It is clearly true for
            $S=\emptyset$.  Now let $S\neq\emptyset$ and let $a\in S$, $T=S
            \setminus \{a\}$.  Then $\lam v_S$ is the sum of weights of the
            neighbours of $S$, and in particular $\lam v_S\ge v_T\ge
            v_\emptyset\lam^{-|T|}$.

            Since we assume $l >i$, the set $[d]^{(\le i)} \setminus V(G)$ is
            non empty. Pick a minimal element $S$ in it.  By moving the weight
            $v_{[l]}$ from $[l]$ to $S$, the value of $\langle A(G)v,v\rangle$
            decreases by $v_{[l]}\cdot(2\sum\limits_{j\in[l]}v_{[l] \setminus
            \{j\}}-2\sum\limits_{j\in S}v_{S \setminus \{j\}})$.  This
            amount is non negative, because $G$ has the largest maximal 
            eigenvalue among subgraphs of $Q_d$ with the same number of
            vertices. Hence,
            \begin{align*}
                \lam v_{[l]} = 
                \sum\limits_{j\in[l]}v_{[l] \setminus \{j\}} \ge 
                \sum\limits_{j\in S}v_{S \setminus \{j\}} \ge 
                v_{\emptyset}\lam^{-(|S| - 1)} \ge 
                v_{\emptyset}\lam^{-(i - 1)}.
            \end{align*}
            and Claim \ref{claim:val-of-l} follows.
        \end{proof}

        In the following claim we conclude that
        $v_{[l-i]}=\Omega({v_\emptyset/l^i})=\Omega({v_\emptyset/\log^i d})$.
        \begin{claim} \label{claim:estimate-v-l-i}
            $v_{[l - i]} = \Omega\big(v_{[l]}(\frac{\lam}{l})^i \big) = 
            \Omega({v_\emptyset/l^i})$.
        \end{claim}
        \begin{proof}
            We shall show that $v_{[l - j - 1]} = \Omega (\lam v_{[l - j ]} /
            l)$ for every $0\le j<i$.  Claim \ref{claim:estimate-v-l-i} then
            follows from Claim \ref{claim:val-of-l}.

            Assume that this assertion does not hold and denote $t= \min \{j:
            v_{[l - j - 1]} = o(\lam v_{[l - j]} / l) \}$.  Define for $j
            \in [t]$
            \begin{align*}
                & \A_j=\big\{S\in V(G)\cap [d]^{(l-t+j)}:[l-t]\subseteq S \big\}\\
                & \B_j=\big\{T \in \big(V(G)\cap [d]^{(l-t+j-1)}\big) \setminus
                \A_{j-1}: \text{there exists } S \in \A_j \text{ s.t. }T
            \subseteq S   \big\}\\
                &W_j=\sum\limits_{S\in \A_j}v_S \qquad\qquad U_j=\sum\limits_{S\in \B_j}v_S.
            \end{align*}

            We obtain the following inequalities, using the fact that every
            vertex in $\A_j$ has at most $d - j$ up-neighbours in $\A_{j + 1}$.
            \begin{align*}
            \lam W_j\le 
            \left\{
            \begin{array}{l l}
            W_1+U_0 & j=0\\
            (d-j+1)W_{j-1}+(j+1)W_{j+1}+U_j & 0<j<t\\
            (d-t+1)W_{t-1}+U_t &j=t\\
            \end{array}
            \right.
            \end{align*}

            Define $k=\min \{j: W_{j+1}=o(\lam W_j)\}$.  Clearly $k\le t$ since
            $W_{t+1}=0$ (by the definition of $l$).  Thus
            $W_j=\Omega(\lam^jW_0)$ for $0\le j\le k$.

            Note that $U_0\le lv_{[l-t-1]}=o(\lam v_{[l-t]}) = o(\lam W_0)$, by
            our assumption.  Also, $\lam U_j\ge (j+1)U_{j+1}$ for $0\le j\le t$,
            thus $U_j=O(\lam^jU_0)=o(\lam^{j+1}W_0)=o(\lam W_j)$ for $0\le j\le
            k$.  Hence, the above inequalities can be rewritten as follows.
            \begin{align*}
            \lam W_j\le 
            \left\{
            \begin{array}{l l}
            W_1+o(\lam W_0) & j=0\\
            (d-j+1)W_{j-1}+(j+1)W_{j+1}+o(\lam W_j) & 0<j<k\\
            (d-k+1)W_{k-1}+o(\lam W_k) &j=k\\
            \end{array}
            \right.
            \end{align*}

            Denote $W=(W_0,\ldots, W_k)^T$, and let $A$ be the matrix with the
            above coefficients with the $o(\lam W_j)$ terms dropped.  The above
            inequalities translate to $\lam W_j\le (AW)_j+o(\lam W_j)$.  We
            obtain the following chain of inequalities, where $X_j=o(\lam W_j)$.
            \begin{align*}
            \lam_1(A)\ge 
            \frac{\langle AW,W\rangle}{\langle W,W \rangle } \ge \frac{\langle \lam W,W\rangle}{\langle W,W \rangle }-\frac{\langle X,W\rangle}{\langle W,W \rangle } =\lam -o(\lam)\ge \lam_1(H_d^i)-o(\sqrt{d}).
            \end{align*}

            However, $\lam_1(A)=\lam_1(H_d^{k})$, thus we obtained
            $\lam_1(H_d^{i})-\lam_1(H_d^k)=o(\sqrt{d})$ for some $0 \le k < i$,
            which is a contradiction to Lemma \ref{lem:calc-evals-ham-balls}.
        \end{proof}

        Claim \ref{claim:estimate-v-l-i} implies that
        $v_{[l-i]}=\omega(v_{[l-i-1]}\lam /l)$ (otherwise, $v_{[l - i]} >
        v_{\emptyset}$, a contradiction to the assumption that $v$ is
        compressed).  Similarly to the proof of the claim, denote by $X_j$ the
        sum of weights of vertices in the $(l-i+j)^{\text{th}}$ level containing $[l-i]$.
        By the same arguments we obtain a contradiction unless $X_j=\Omega(\lam
        X_{j-1})$ for $1\le j\le i$.  Thus $X_i=\Omega(X_0 \lam^i)=\Omega
        (v_\emptyset \lam^i/\log ^id)$.  Recall that by Claim \ref{claim:theta},
        there are at most $O(d^{i-\theta})$ vertices in $V\cap [d]^{(l)}$.
        Hence, using the fact that $v$ is compressed, $v_{[l]}=
        \Omega(d^{-i/2+\theta/2})$.  By Claim \ref{claim:val-of-l}, this implies
        that
        $v_{[l-i]}=\Omega(v_{[l]}d^{\theta/2}/\log^id)=\omega(v_\emptyset)$,
        contradicting the assumption that $v$ is compressed.  Hence we cannot
        have $l>i$, so $G=H_d^i$, as required.  This proves that $H_d^i$
        maximises $\lam_1$ among subgraphs of the cube with the same number of
        vertices, thus completing the proof of our second main result, Theorem
        \ref{thm:fixed-radius}.
    \end{proof}

\section{Conclusion} \label{sec:conclusion}
    The question of characterising the subgraphs of the cube maximising $\lam_1$ is
    far from being completely answered.  For radii tending to infinity with the
    dimension of the cube, our results as well as  Samorodnitsky's results
    \cite{samo} only show that the Hamming balls have largest eigenvalues which
    are asymptotically largest among subgraphs of the same order.  We believe
    that, similarly to Theorem \ref{thm:star}, the Hamming balls maximise the
    maximal eigenvalues exactly rather than just asymptotically, for large $d$ and a
    large range of radii.
    \begin{question}
        Is it true that if $d/2 - i$ is sufficiently large, then $H_d^i$
        maximises $\lam_1$ among subgraphs of $Q_d$ with the same number of
        vertices?
    \end{question}

    We point out that for radii that are very close to $d/2$ the Hamming ball
    does not achieve the largest maximal eigenvalue, as can be seen by the following
    example.
    \begin{ex}
        Assume that $d$ is even and consider the Hamming ball of radius $d/2-1$,
        $H=H_d^{d/2-1}$.  We show that $\lambda_1(H)=d-2$.  Put $\lambda=d-2$
        and let $x$ be the vector with weight $x_i = 1-2i/d$ on the vertices of
        level $i$.  The following can be easily verified.
        \begin{align*}
        \lambda x_i=
        \left\{
        \begin{array}{l l}
        dx_1 & i=0\\
        ix_{i-1}+(d-i)x_{i+1} & 0<i<d/2-1\\
        ix_{i-1} &i=d/2-1\\
        \end{array}
        \right.
        \end{align*}
        Thus we have $A(H)x=(d-2)x$.  Since all the weights $x_i$ are positive,
        this implies that $\lam_1(H)=d-2$.  Note that
        $|H|=2^{d-1}(1-\Theta(1/\sqrt{d}))>2^{d-2}$. Thus, since the largest
        eigenvalue of the subcube of dimension $d-2$ is $d-2$, we can achieve a
        larger maximal eigenvalue with a (connected) subgraph with $|H|$
        vertices containing the subcube of dimension $d-2$.

        Similarly, it can be shown that if $i = d/2 - \sqrt{d}$ is an integer,
        then $\lam_1(H_d^i) = d - 4$. Since $|H_d^i| > 2^{d - 4}$, also in this
        case the largest eigenvalue of the Hamming ball is not maximal among
        subgraphs of the cube of the same order.
    \end{ex}

    As seen by this example, it may be interesting to consider subgraphs whose
    largest eigenvalue is very close to $d$. For instance, determining the range
    of radii for which the Hamming balls maximise the maximal eigenvalue,
    especially for large radii, seems like a challenging problem.  The following
    weaker problem also seems hard. Is it true that for every fixed $c>0$, if a
    subgraph $H$ of $Q_d$ has $\lam_1(H)\ge d-c$, then $|H|= \Omega(2^{d})$?

    \bibliography{eigenvalue}
    \bibliographystyle{amsplain}
    
\end{document}